\documentclass[12pt,twoside,reqno]{amsart}
\usepackage{amsmath}
\usepackage{amsfonts}
\usepackage{amssymb}
\usepackage{color}
\usepackage{mathrsfs}
\usepackage{cite}
\usepackage{geometry}
\usepackage{marginnote}
\usepackage{todonotes}
\newtheorem{theorem}{Theorem}[section]
\newtheorem{corollary}[theorem]{Corollary}
\newtheorem{lemma}[theorem]{Lemma}
\newtheorem{proposition}[theorem]{Proposition}
\newtheorem{definition}[theorem]{Definition}
\newtheorem{remark}[theorem]{Remark}
\numberwithin{equation}{section}
\begin{document}
\title{Mass-conserving solutions to coagulation-fragmentation equations with non-integrable fragment distribution function} 

\author{Philippe Lauren\c{c}ot}
\address{Institut de Math\'ematiques de Toulouse, UMR~5219, Universit\'e de Toulouse, CNRS \\ F--31062 Toulouse Cedex 9, France}
\email{laurenco@math.univ-toulouse.fr}


\date{\today}

\begin{abstract}
Existence of mass-conserving weak solutions to the coagulation-frag\-men\-ta\-tion equation is established when the fragmentation mechanism produces an infinite number of fragments after splitting. The coagulation kernel is assumed to increase at most linearly for large sizes and no assumption is made on the growth of the overall fragmentation rate for large sizes. However, they are both required to vanish for small sizes at a rate which is prescribed by the (non-integrable) singularity of the fragment distribution.
\end{abstract}

\maketitle

%
%
\pagestyle{myheadings}
\markboth{\sc{Ph. Lauren\c cot}}{\sc{C-F equations with non-integrable fragment distribution}}

\section{Introduction}\label{s1}

A mean-field description of the dynamics of a system of particles varying their sizes by pairwise coalescence and multiple fragmentation is provided by the coagulation-fragmentation equation
\begin{subequations}\label{a1}
\begin{equation}
\partial_t f(t,x) = \mathcal{C}f(t,x) + \mathcal{F}f(t,x)\ , \qquad (t,x)\in (0,\infty)^2\ , \label{a1a}
\end{equation}
where the coagulation and fragmentation reaction terms are given by
\begin{equation}
\mathcal{C}f(x) := \frac{1}{2} \int_0^x K(x-y,y) f(y) f(x-y)\ \mathrm{d}y - \int_0^\infty K(x,y) f(x) f(y)\ \mathrm{d}y \label{a1b}
\end{equation}
and
\begin{equation}
\mathcal{F}f(x) := - a(x) f(x) + \int_x^\infty a(y) b(x,y) f(y)\ \mathrm{d}y \label{a1c}
\end{equation}
for $x\in (0,\infty)$. 
\end{subequations}
We supplement \eqref{a1} with an initial condition 
\begin{equation}
f(0,x) = f^{in}(x)\ , \qquad x\in (0,\infty)\ . \label{a2}
\end{equation}
In  \eqref{a1}, $f=f(t,x)$ denotes the size distribution function of the particles of size $x\in (0,\infty)$ at time $t>0$ and $K$ is the coagulation kernel which describes how likely particles of respective sizes $x$ and $y$ merge. The first term in the coagulation term \eqref{a1b} accounts for the formation of particles of size $x$ resulting from the coalescence of two particles of respective sizes $y\in (0,x)$ and $x-y$, while the second term describes the disappearance of particles of size $x$ as they merge with other particles of arbitrary size. The first term in the fragmentation term \eqref{a1c} involves the overall fragmentation rate $a(x)$ and accounts for the loss of particles of size $x$ due to breakup while the second term describes the contribution to particles of size $x$ of the fragments resulting from the splitting of a particle of arbitrary size $y>x$, the distribution of fragments being given according to the daughter distribution function $b(x,y)$. Since no matter is lost during fragmentation events, $b$ is required to satisfy
\begin{equation}
\int_0^y x b(x,y)\ \mathrm{d}x = y\ , \qquad y>0\ . \label{a3}
\end{equation}
As the merging of two particles of respective sizes $x$ and $y$ results in a particle of size $x+y$, it is then expected that conservation of matter holds true during time evolution, that is,
\begin{equation}
\int_0^\infty x f(t,x)\ \mathrm{d}x = \varrho := \int_0^\infty x f^{in}(x) \mathrm{d}x\ , \qquad t\ge 0\ , \label{a4}
\end{equation}
provided $\varrho$ is finite. Since the pioneering works by Melzak \cite{Melz57b}, McLeod \cite{McLe64}, and Stewart \cite{Stew89}, the existence of weak solutions to \eqref{a1}-\eqref{a2} is investigated in several papers under various assumptions on the coagulation and fragmentation coefficients $K$, $a$, and $b$ and the initial condition $f^{in}$ \cite{BaLa11, BaLa12a, BLL13, DuSt96b, ELMP03, GKW11, GLW12, Laur00, LaMi02b}, see also \cite{BaCa90, daCo95, Laur02, LeTs81, Whit80} for the discrete coagulation-fragmentation equation. Assuming in particular the initial condition $f^{in}$ to be a non-negative function in $L^1((0,\infty),x\mathrm{d}x)$, the solutions to \eqref{a1}-\eqref{a2} constructed in the above mentioned references are either mass-conserving, i.e. satisfy \eqref{a4}, or not, according to the growth of the coagulation kernel $K$ for large sizes or the behaviour of the overall fragmentation rate $a$ for small sizes. Indeed, it is by now well-known that, if the growth of the coagulation kernel for large sizes is sufficiently fast (such as $K(x,y) \ge K_0 (xy)^{\lambda/2}$ for some $\lambda>1$ and $K_0$), then a runaway growth takes place in the system of particles and a \textsl{giant particle} (or particle of infinite size) is formed in finite time, a phenomenon usually referred to as \textsl{gelation} \cite{ELMP03, EMP02, HEZ83, LeTs81, Leyv83}. Since all particles accounted for by the size distribution have finite size, the occurrence of gelation results in a loss of matter in the dynamical behaviour of \eqref{a1}. A somewhat opposite phenomenon takes place when the overall fragmentation rate $a$ blows up as $x\to 0$ (such as $a(x)=a_0 x^\gamma$ for some $\gamma<0$ and $a_0>0$). In that case, the smaller the particles, the faster they split, leading to the instantaneous appearance of \textsl{dust} (or particles of size zero) and again a loss of matter takes place, usually referred to as the \textsl{shattering} transition \cite{ArBa04, Fili61, McZi87}. We shall exclude the occurrence of these phenomena in the forthcoming analysis and focus on a different feature of the fragmentation mechanism, namely the possibility that an infinite number of fragments is produced during breakup. Specifically, a common assumption on the daughter distribution function $b$ in the above mentioned references is its integrability, that is, 
\begin{equation}
n_0(y) := \int_0^y b(x,y)\ \mathrm{d}x < \infty \;\text{ for almost every }\; y>0\ , \label{a5}
\end{equation}
which amounts to require that the splitting of a particle of size $y$ produces $n_0(y)$ daughter particles, a particularly important case being the so-called binary fragmentation corresponding to $n_0(y)=2$ for all $y>0$. Clearly, the integrability property \eqref{a5} fails to be true when $b$ is given by \cite{McZi87}
\begin{equation}
b_\nu(x,y) = (\nu+2) \frac{x^\nu}{y^{\nu+1}}\ , \qquad 0<x<y\ , \qquad \nu\in (-2,-1]\ . \label{a6}
\end{equation}
Observe that the restriction $\nu>-2$ guarantees that $b_\nu$ satisfies \eqref{a3}. 

As far as we know, the existence of weak solutions to the coagulation-fragmentation equation \eqref{a1}-\eqref{a2} when the daughter distribution function $b$ is given by \eqref{a6} and the coagulation kernel is unbounded has not been considered so far, except the particular case $K(x,y)=xy$, $a(x)=x$, and $\nu=-1$ which is handled in \cite{LavR15} by an approach exploiting fully the  specific structure of the coefficients. The purpose of this note is to fill this gap, at least for coagulation kernels growing at most linearly for large sizes. In fact, the main difficulty to be overcome is the following: due to the non-integrability of $b_\nu$, the natural functional setting, which is the space $L^1((0,\infty),(1+x)\mathrm{d}x)$, can no longer be used. As we shall see below, it might be replaced by the smaller space $L^1((0,\infty),(x^m+x)\mathrm{d}x)$ for some $m\in (0,1)$ such that
\begin{equation*}
\int_0^y x^m b(x,y)\ \mathrm{d}x = (\nu+2) y^m \int_0^1 z^{m+\nu}\ \mathrm{d}z< \infty\ , \qquad y>0\ ,
\end{equation*}
that is, $m>-1-\nu\ge 0$. This choice requires however the coagulation kernel $K$ and the overall fragmentation rate $a$ to vanish in an appropriate way for small sizes. More precisely, we assume that there are $K_0>0$ and $m_0\in (-1-\nu,1)$ such that
\begin{subequations}\label{a7}
\begin{equation}
0 \le K(x,y) = K(y,x) \le K_0 (2+x+y)\ , \qquad (x,y)\in (0,\infty)^2\ , \label{a7a}
\end{equation}
and, for all $R>0$, 
\begin{equation}
L_R := \sup_{(x,y)\in (0,R)^2} \left\{ \frac{K(x,y)}{\min\{x,y\}^{m_0}} \right\} < \infty\ . \label{a7b}
\end{equation}
\end{subequations}
We also assume that, for all $R>0$, there is $A_R>0$ such that
\begin{equation}
a(x) \le A_R x^{m_0+\nu+1}\ , \qquad x\in (0,R)\ . \label{a8}
\end{equation}
Roughly speaking, $K$ and $a$ are required to vanish faster for small sizes as the singularity of $b_\nu$. For instance, the coagulation kernel $K$ and the homogeneous overall fragmentation rate $a$ given by
\begin{equation}
K(x,y)=x^\alpha y^\beta +x^\beta y^\alpha\ , \qquad a(x)=x^\gamma\ , \qquad (x,y)\in (0,\infty)^2\ , \label{exKa}
\end{equation} 
with $-1-\nu<\alpha\le\beta\le 1-\alpha$ and $\gamma>0$ satisfy \eqref{a7} and \eqref{a8}, respectively, for any $m_0\in \left( -1-\nu, \max\{ \alpha , \gamma-1-\nu\} \right]$. Observe in particular that no growth constraint is required on $a$ for large sizes. 

Before stating our results, let us introduce some notation: given $m\in \mathbb{R}$, we define the space $X_m$ and its positive cone $X_m^+$ by
\begin{equation*}
X_m := L^1((0,\infty),x^m\mathrm{d}x)\ , \qquad X_m^+ := \{ f\in X_m\ :\ f\ge 0 \;\text{ a.e. in }\; (0,\infty) \} \ , 
\end{equation*}
respectively, and denote the space $X_m$ endowed with its weak topology by $X_{m,w}$. We also put
\begin{equation*}
M_m(f) := \int_0^\infty x^m f(x)\ \mathrm{d}x\ , \qquad f\in X_m\ .
\end{equation*}

We begin with the existence of mass-conserving weak solutions to \eqref{a1}-\eqref{a2} when the initial condition $f^{in}$ lies in $X_{m_0}^+\cap X_1$ for some $m_0\in (-\nu-1,1)$ and decays in a suitable way for large sizes.

\begin{theorem}\label{thmi1}
Let $\nu\in (-2,-1]$. Assume that the coagulation kernel $K$ and the overall fragmentation rate $a$ satisfy \eqref{a7} and \eqref{a8}, respectively, and that the daughter distribution function $b=b_\nu$ is given by \eqref{a6}. Given an initial condition $f^{in}\in X_{m_0}^+\cap X_1$ such that
\begin{equation}
\int_0^\infty x \ln(\ln(x+5)) f^{in}(x)\ \mathrm{d}x < \infty\ , \label{a9}
\end{equation}
there exists at least one weak solution $f\in C([0,\infty); X_{m_0,w}\cap X_{1,w})$ on $[0,\infty)$ to \eqref{a1}-\eqref{a2} such that 
\begin{equation}
M_1(f(t)) = M_1(f^{in})\ , \qquad t\ge 0\ . \label{a10}
\end{equation}
Moreover, if $f^{in}\in X_m$ for some $m>1$, then $f\in L^\infty(0,T;X_m)$ for all $T>0$.
\end{theorem}

When $b=b_\nu$ and $\nu>-1$, the existence of mass-conserving weak solutions on $[0,\infty)$ to \eqref{a1}-\eqref{a2} is already known for $f^{in}\in X_0^+\cap X_1$, the assumptions \eqref{a7b} and \eqref{a8} being replaced by the local boundedness of $K$ and $a$  \cite{LaMi02b, Stew89}. In particular, it does not require the additional integrability condition \eqref{a9} on $f^{in}$ and it is yet unclear whether Theorem~\ref{thmi1} is valid under the sole assumption $f^{in}\in X_0^+\cap X_1$. In fact, the condition \eqref{a9} can be replaced by $f^{in}\in L^1((0,\infty),w(x)\mathrm{d}x)$ for any weight function $w$ enjoying the properties listed in Lemma~\ref{lemb3} below, which includes weights involving multiple iterates of the logarithm function. We nevertheless restrict our analysis to the specific choice of the weight function in \eqref{a9} as we have not yet identified an ``optimal'' class of initial conditions for which Theorem~\ref{thmi1} is true.

The proof of Theorem~\ref{thmi1} proceeds along the lines of the existence proofs performed in \cite{LaMi02b, Stew89}, the main differences lying in the control of the behaviour for small sizes and its consequences on the behaviour for large sizes. It relies on a weak compactness method in $X_{m_0}$, the starting point being the construction of a truncated version of \eqref{a1} for which well-posedness can be established by a classical Banach fixed point argument. The compactness approach involves five steps: we first show that the additional integrability assumption \eqref{a9} is preserved throughout time evolution and additionally provides a control on the behaviour of the fragmentation term for large sizes whatever the growth of the overall fragmentation rate $a$. We next turn to the behaviour for small sizes and prove boundedness in $X_{m_0}$, the outcome of the previous step being used to control the contribution of the fragmentation gain term. The third and fourth steps are more classical and devoted to the uniform integrability and the time equicontinuity in $X_{m_0}$, respectively. In the last step, we gather the outcome of the four previous ones to show the convergence of the solutions of the truncated problem as the truncation parameter increases without bound and that the limit thus obtained is a weak solution on $[0,\infty)$ to \eqref{a1}-\eqref{a2} which satisfies the conservation of matter \eqref{a10}. We finally use the same argument as in the first step to prove the stability of $X_m$ for $m>1$.

\begin{remark}\label{remi3}
For simplicity, the analysis carried out in this paper is restricted to the daughter distribution functions given by \eqref{a6}. The proof of Theorem~\ref{thmi1} can however be adapted to encompass a broader class of daughter distribution functions $b$ featuring a non-integrable singularity for  small fragment sizes, see \cite{BLLxx}. In the same vein, existence of weak solutions (not necessarily mass-conserving) can be proved by a similar approach when the coagulation kernel $K$ features a faster growth than \eqref{a7a} and we refer to \cite{BLLxx} for a more complete account.
\end{remark}

We supplement the existence result with a uniqueness result, which is however only valid for a smaller class of coagulation kernels $K$ and initial conditions $f^{in}$.

\begin{theorem}\label{thmi2}
Let $\nu\in (-2,-1]$ and $\delta\in (0,1)$. Assume that the coagulation kernel $K$ and the overall fragmentation rate $a$ satisfy \eqref{a7} and \eqref{a8}, respectively, and that the daughter distribution function $b=b_\nu$ is given by \eqref{a6}. Assume also that there is $K_1>0$ such that
\begin{equation}
K(x,y)\le K_1 x^{m_0} y\ , \qquad (x,y)\in (0,1)\times (1,\infty)\ . \label{u0}
\end{equation}
Given an initial condition $f^{in}\in X_{m_0}^+\cap X_{2+\delta}$, there is a unique mass-conserving weak solution $f$ on $[0,\infty)$ to \eqref{a1}-\eqref{a2} such that $f\in L^\infty(0,T;X_{2+\delta})$ for all $T>0$.
\end{theorem}

Observing that any function $f^{in}\in X_{m_0}^+\cap X_{2+\delta}$ satisfies \eqref{a9}, the existence part of Theorem~\ref{thmi2} readily follows from Theorem~\ref{thmi1}. As Theorem~\ref{thmi1}, Theorem~\ref{thmi2} applies to \eqref{a1}-\eqref{a2} when the coagulation kernel $K$ and the overall fragmentation rate $a$ are given by \eqref{exKa} with $-1-\nu < \alpha \le \beta \le 1-\alpha$, $\gamma>0$, and $m_0\in \left( -1-\nu, \max\{ \alpha , \gamma-1-\nu\} \right]$. We also mention that, when $b=b_\nu$ and $\nu>-1$, Theorem~\ref{thmi2} is valid without the assumption \eqref{u0} on $K$ and for $f^{in}\in X_0^+ \cap X_2$ \cite{LaMi04, Norr99}.

As in \cite{EMRR05, LaMi04, Norr99, Stew90b}, the uniqueness proof relies on a control of the distance between two solutions in a weighted $L^1$-space, the delicate point being the choice of an appropriate weight which turns out to be $\xi(x):=\max\{ x^{m_0} , x^{1+\delta}\}$, $x>0$, here. The superlinearity of $\xi$ for large sizes compensates the sublinearity of $\xi$ for small sizes which gives a positive contribution of the fragmentation term.

\section{Existence}\label{s2}

\newcounter{NumConst}

Throughout this section, the parameter $\nu\in (-2,-1]$ is fixed and we assume that the coagulation kernel $K$ and the overall fragmentation rate $a$ satisfy \eqref{a7} and \eqref{a8}, respectively, while the daughter distribution function $b=b_\nu$ is given by \eqref{a6}. Also, $f^{in}$ is a function in $X_{m_0}^+\cap X_1$ enjoying the additional integrability property \eqref{a9} and we set $\varrho := M_1(f^{in})$.

\subsection{Preliminaries}\label{s2.1}

Let us first begin with the definition of a weak solution to \eqref{a1}-\eqref{a2}. 

\begin{definition}\label{defb0}
Let $T\in (0,T]$. A weak solution on $[0,T)$ to \eqref{a1}-\eqref{a2} is a non-negative function $f\in C([0,\infty); X_{m_0,w}\cap X_{1,w})$ on $[0,\infty)$ such that, for all $t\in (0,T)$ and $\vartheta\in\Theta_{m_0}$, 
\begin{align}
\int_0^\infty (f(t,x)-f^{in}(x)) \vartheta(x)\ \mathrm{d}x & = \frac{1}{2} \int_0^\infty \int_0^\infty K(x,y) \chi_\vartheta(x,y) f(t,x) f(t,y)\ \mathrm{d}y\mathrm{d}x \nonumber\\
& \qquad - \int_0^\infty a(x) N_\vartheta(x) f(t,x)\ \mathrm{d}x\ , \label{b0}
\end{align}
where
\begin{align*}
\chi_\vartheta(x,y) & := \chi(x+y) -\chi(x) - \chi(y)\ , \qquad (x,y)\in (0,\infty)^2\ , \\
N_\vartheta(y) & := \vartheta(y) - \int_0^y \vartheta(x) b_\nu(x,y)\ \mathrm{d}x\ , \qquad y>0\ ,
\end{align*}
and
\begin{equation*}
\Theta_{m_0} := \left\{ \vartheta\in C^{m_0}([0,\infty))\cap L^\infty(0,\infty)\ :\ \vartheta(0) = 0 \right\}\ .
\end{equation*}
\end{definition}

Observe that, for $\vartheta\in \Theta_{m_0}$ and $(x,y)\in (0,\infty)^2$, 
\begin{equation*}
|\chi_\vartheta(x,y)| \le 2\|\vartheta\|_{C^{m_0}} \min\{x,y\}^{m_0} \;\text{ and }\; |N_\vartheta(y)| \le \left( 1 + \frac{\nu + 2}{\nu+m_0+1} \right) \|\vartheta\|_{C^{m_0}} y^{m_0} \ , 
\end{equation*}
so that the right-hand side of \eqref{b0} is well-defined.

We next define a particular class of convex functions on $[0,\infty)$ which proves useful in the forthcoming analysis.

\begin{definition}\label{defb1}
	A non-negative and convex function $\varphi\in C^\infty([0,\infty))$ belongs to the class $\mathcal{C}_{VP,\infty}$ if it satisfies the following properties:
	\begin{itemize}
		\item[(a)] $\varphi(0)=\varphi'(0)=0$ and $\varphi'$ is concave;
		\item[(b)] $\lim_{r\to\infty} \varphi'(r) = \lim_{r\to\infty} \varphi(r)/r=\infty$;
		\item[(c)] for $p\in (1,2)$, 
		\begin{equation*}
		S_p(\varphi) := \sup_{r\ge 0}\left\{ \frac{\varphi(r)}{r^p} \right\} < \infty\ .
		\end{equation*}
	\end{itemize}
\end{definition}

For instance, $r\mapsto (r+1)\ln(r+1)-r$ lies in $\mathcal{C}_{VP,\infty}$. Functions in the class $\mathcal{C}_{VP,\infty}$ enjoy several properties which we list now.

\begin{lemma}\label{lemb2}
	Let $\varphi\in \mathcal{C}_{VP,\infty}$. Then
	\begin{itemize}
		\item[(a)] $r\mapsto \varphi(r)/r$ is concave on $[0,\infty)$;
		\item[(b)] for $r\ge 0$, 
		\begin{equation*}
		0 \le \varphi(r) \le r\varphi'(r) \le 2 \varphi(r) \ ;
		\end{equation*}
		\item[(c)] for $(r,s)\in [0,\infty)^2$, 
		\begin{equation}
		s \varphi'(r)  \le \varphi(r) + \varphi(s)\ , \label{x1}
		\end{equation}
		and
		\begin{align}
		0 \le \varphi(r+s) - \varphi(r) - \varphi(s) & \le 2 \frac{s \varphi(r) + r \varphi(s)}{r+s}\ , \label{x2} \\
		0 \le \varphi(r+s) - \varphi(r) - \varphi(s) & \le \varphi''(0) rs \label{x3}\ .
		\end{align}
	\end{itemize}
\end{lemma}

\begin{proof}
Since 
\begin{equation*}
\varphi(r+s) - \varphi(r) - \varphi(s) = \int_0^r \int_0^s \varphi''(r_*+s_*)\ \mathrm{d}s_* \mathrm{d}r_*\ , \qquad (r,s)\in [0,\infty)^2\ ,
\end{equation*}
the inequality \eqref{x3} readily follows from the concavity of $\varphi'$. The other properties listed in Lemma~\ref{lemb2} are consequences of the convexity of $\varphi$ and the concavity of $\varphi'$ and are proved, for instance, in \cite[Proposition~14]{Laur15}.
\end{proof}

We finally collect some properties of the weight involved in the assumption \eqref{a9}.
\begin{lemma}\label{lemb3}
Define $W(x) := x \ln(\ln(x+5)) - x \ln(\ln(5))$ for $x\ge 0$. Then $W\in \mathcal{C}_{VP,\infty}$ and, for all $m\in [0,1)$,
\begin{equation*}
\lim_{x\to\infty} \frac{x W'(x)-W(x)}{x^m} = \infty\ .
\end{equation*}
\end{lemma}

\subsection{Approximation}\label{s2.2}

Let $j\ge 2$ be an integer and set
\begin{equation}
K_j(x,y) := K(x,y) \mathbf{1}_{(0,j)}(x+y)\ , \qquad a_j(x) := a(x) \mathbf{1}_{(0,j)}(x)\ , \label{b2} 
\end{equation}
and
\begin{equation}
f_j^{in}(x) := f^{in}(x) \mathbf{1}_{(0,j)}(x)\label{b3}
\end{equation}
for $(x,y)\in (0,\infty)^2$. We denote the coagulation and fragmentation operators with $K_j$ and $a_j$ instead of $K$ and $a$ by $\mathcal{C}_j$ and $\mathcal{F}_j$, respectively. Simple computations along with \eqref{a7b} and the boundedness of $a_j$ resulting from \eqref{a8} show that both $\mathcal{C}_j$ and $\mathcal{F}_j$ are locally Lipschitz continuous from $L^1((0,j),x^{m_0}\mathrm{d}x)$ into itself \cite{BLLxx}. Thanks to these properties, we are in a position to apply the Banach fixed point theorem to prove that there is a unique non-negative function 
\begin{equation*}
f_j\in C^1([0,\infty),L^1((0,j),x^{m_0}\mathrm{d}x)) 
\end{equation*}
solving
\begin{subequations}\label{b5}
	\begin{align}
	\partial_t f_j(t,x) & = \mathcal{C}_j f_j(t,x) + \mathcal{F}_j f_j(t,x)\ , \qquad (t,x) \in (0,\infty)\times (0,j)\ , \label{b5a} \\
	f_j(0,x) & = f_j^{in}(x)\ , \qquad x\in (0,j)\ . \label{b5b}
	\end{align}
\end{subequations}
Introducing the space $\Theta_{m_0,j} := \{ \vartheta\in C^{m_0}([0,j])\ :\ \vartheta(0)=0\}$, it readily follows from \eqref{b2} and \eqref{b5} that $f_j$ satisfies the following weak formulation of \eqref{b5}: for all $\vartheta\in \Theta_{m_0,j}$ and $t>0$,
\begin{align}
\frac{\mathrm{d}}{\mathrm{d}t} \int_0^j \vartheta(x) f_j(t,x)\ \mathrm{d}x & = \frac{1}{2} \int_0^j \int_0^{j-x} K(x,y) \chi_\vartheta(x,y) f_j(t,x) f_j(t,y)\ \mathrm{d}y\mathrm{d}x \nonumber\\
& \qquad - \int_0^j a(x) N_\vartheta(x) f_j(t,x)\ \mathrm{d}x\ , \label{b6}
\end{align} 
the functions $\chi_\vartheta$ and $N_\vartheta$ being defined in Definition~\ref{defb0}. Choosing $\vartheta(x)=x$, $x\in (0,j)$, in \eqref{b6} readily gives the conservation of matter
\begin{equation*}
\int_0^j x f_j(t,x)\ \mathrm{d}x = \int_0^j x f^{in}(x)\ \mathrm{d}x\ , \qquad t\ge 0\ .
\end{equation*}
Extending $f_j$ to $[0,\infty)\times (j,\infty)$ by zero (i.e. $f_j(t,x)=0$ for $(t,x)\in [0,\infty)\times (j,\infty)$), the previous identity reads
\begin{equation}
M_1(f_j(t)) = M_1(f_j^{in}) \le \varrho = M_1(f^{in})\ , \qquad t\ge 0\ . \label{b7}
\end{equation}

In addition, introducing
\begin{equation}
C_0 := M_{m_0}(f^{in}) + \int_0^\infty W(x) f^{in}(x)\ \mathrm{d}x < \infty\ , \label{b8}
\end{equation}
which is finite according to the integrability properties of $f^{in}$ and in particular \eqref{a9}, we infer from \eqref{b3}, \eqref{b8}, and the non-negativity of $W$ that 
\begin{equation}
M_{m_0}(f_j^{in}) + \int_0^\infty W(x) f_j^{in}(x)\ \mathrm{d}x \le C_0\ . \label{b9}
\end{equation}

We now investigate the weak compactness features of the sequence $(f_j)_{j\ge 2}$. In the sequel, $C$ and $C_i$, $i\ge 1$, denote positive constants depending on $K$, $a$, $\nu$, $m_0$, $f^{in}$, $\varrho$, and $C_0$. Dependence upon additional parameters is indicated explicitly. 

\subsection{Moment Estimates}\label{s2.3}

We start with a control on the behaviour for large sizes. 

\begin{lemma}\label{lemb4}
\refstepcounter{NumConst}\label{cst1}
Let $T>0$. There is $C_{\ref{cst1}}(T)>0$ depending on $T$ such that, for $t\in [0,T]$, 
\begin{equation*}
\int_0^\infty W(x) f_j(t,x)\ \mathrm{d}x + \int_0^t \int_0^\infty a(x) [x W'(x) - W(x)] f_j(s,x)\ \mathrm{d}x\mathrm{d}s  \le C_{\ref{cst1}}(T)\ .
\end{equation*}
\end{lemma}

\begin{proof}
On the one hand, since $W\in \mathcal{C}_{VP,\infty}$ by Lemma~\ref{lemb3}, we infer from \eqref{a7a}, \eqref{x2}, and \eqref{x3} that, 
\begin{equation}
K(x,y) \chi_W(x,y) \le 2K_0 W''(0) xy + 2K_0 [xW(y) + yW(x)] \label{b10}
\end{equation}
for $(x,y)\in (0,\infty)^2$. On the other hand, the function $W_1: x \mapsto W(x)/x$ is concave according to Lemma~\ref{lemb2}~(a) and it follows from \eqref{a6} that, for $y>0$,
\begin{align}
N_W(y) & = \int_0^y [W_1(y) - W_1(x)] x b_\nu(x,y)\ \mathrm{d}x \ge  \int_0^y W_1'(y) x(y-x) b_\nu(x,y)\ \mathrm{d}x \nonumber\\ 
& \ge \frac{y W'(y) - W(y)}{y^2} \int_0^y x(y-x) b_\nu(x,y)\ \mathrm{d}x = \frac{yW'(y)-W(y)}{\nu+3}\ . \label{b11}
\end{align}
Combining \eqref{b6} with $\vartheta=W$, \eqref{b10}, and \eqref{b11} gives, for $t>0$,
\begin{align*}
& \frac{\mathrm{d}}{\mathrm{d}t} \int_0^\infty W(x) f_j(t,x)\ \mathrm{d}x \\ 
& \qquad \le K_0 \int_0^j \int_0^{j-x} \left[ W''(0)xy + xW(y) + y W(x) \right] f_j(t,x) f_j(t,y)\ \mathrm{d}y\mathrm{d}x \\
& \qquad \qquad - \frac{1}{\nu+3} \int_0^j a(y) [yW'(y)-W(y)] f_j(t,y)\ \mathrm{d}y\ .
\end{align*}
We further deduce from Lemma~\ref{lemb2}~(b) and \eqref{b7} that
\begin{align*}
\frac{\mathrm{d}}{\mathrm{d}t} \int_0^\infty W(x) f_j(t,x)\ \mathrm{d}x & + \frac{1}{2} \int_0^\infty a(y) [yW'(y)-W(y)] f_j(t,y)\ \mathrm{d}y \\
& \le K_0 W''(0)\varrho^2 + 2 K_0 \varrho \int_0^\infty W(x) f_j(t,x)\ \mathrm{d}x\ ,
\end{align*}
Integrating the previous differential inequality and using \eqref{b9} complete the proof.
\end{proof}

Exploiting the properties of $W$ for large sizes along with the outcome of Lemma~\ref{lemb4} provides additional information on the fragmentation term.

\begin{corollary}\label{corb5}
\refstepcounter{NumConst}\label{cst2}
Let $T>0$ and $m\in (0,1)$. There is $C_{\ref{cst2}}(m,T)>0$ depending on $m$ and $T$ such that
\begin{equation*}
\int_0^T P_{m,j}(s) \mathrm{d}s \le C_{\ref{cst2}}(m,T)\ , \qquad P_{m,j}(s) := \int_1^\infty x^m a(x) f_j(s,x)\ \mathrm{d}x\ , \qquad s\in [0,T]\ .
\end{equation*}
\end{corollary}

\begin{proof}
Owing to Lemma~\ref{lemb3}, there is $x_m>1$ such that $xW'(x)-W(x)\ge x^m$ for $x>x_m$. Therefore, by \eqref{a8}, \eqref{b7}, Lemma~\ref{lemb2}~(b), and Lemma~\ref{lemb4},
\begin{align*}
\int_0^T \int_1^\infty x^m a(x) f_j(s,x)\ \mathrm{d}x\mathrm{d}s & \le \int_0^T \int_1^{x_m} x^m a(x) f_j(s,x)\ \mathrm{d}x\mathrm{d}s \\
& \qquad + \int_0^T \int_{x_m}^\infty x^m a(x) f_j(s,x)\ \mathrm{d}x\mathrm{d}s \\
& \le A_{x_m} x_m^{m_0+\nu+1} \int_0^T \int_1^{x_m} x^m f_j(s,x)\ \mathrm{d}x\mathrm{d}s \\
& \qquad + \int_0^T \int_{x_m}^\infty [x W'(x)-W(x)] a(x) f_j(s,x)\ \mathrm{d}x\mathrm{d}s \\
& \le A_{x_m} x_m^{m_0+\nu+1} \varrho T + C_{\ref{cst1}}(T)\ ,
\end{align*}
and the proof is complete.
\end{proof}

We now study the behaviour for small sizes.

\begin{lemma}\label{lemb6}
\refstepcounter{NumConst}\label{cst3}
Let $T>0$. There is $C_{\ref{cst3}}(T)>0$ depending on $T$ such that
\begin{equation*}
M_{m_0}(f_j(t)) \le C_{\ref{cst3}}(T)\ , \qquad t\in [0,T]\ .
\end{equation*}
\end{lemma}

\begin{proof}
Setting $\vartheta_0(x) := \min\{x,x^{m_0}\}$ for $x>0$, we observe that
\begin{equation*}
\chi_{\vartheta_0}(x,y) \le 0\ , \qquad (x,y)\in (0,\infty)^2\ ,
\end{equation*}
and
\begin{align*}
N_{\vartheta_0}(y) & =  - \frac{1-{m_0}}{\nu+{m_0}+1} y^{m_0} \ , & y\in (0,1)\ , \\
N_{\vartheta_0}(y) & = - \frac{1-{m_0}}{\nu+{m_0}+1} \frac{1}{y^{\nu+1}} \ge - \frac{1-{m_0}}{\nu+{m_0}+1} y^{m_0}\ , & y\in (1,\infty)\ .
\end{align*}
We then infer from \eqref{a8} with $R=1$ and \eqref{b6} with $\vartheta=\vartheta_0$ that
\begin{align*}
\frac{\mathrm{d}}{\mathrm{d}t} M_{m_0}(f_j(t)) & \le \frac{1-{m_0}}{\nu+{m_0}+1} \int_0^\infty y^{m_0} a(y) f_j(t,y)\ \mathrm{d}y \\ 
& \le \frac{A_1}{\nu+{m_0}+1} \int_0^1 y^{m_0} f_j(t,y)\ \mathrm{d}y + \frac{P_{m_0,j}(t)}{\nu+{m_0}+1} \\
& \le \frac{A_1}{\nu+{m_0}+1} M_{m_0}(f_j(t)) + \frac{P_{m_0,j}(t)}{\nu+{m_0}+1}\ .
\end{align*}
We integrate the previous differential inequality and complete the proof with the help of Corollary~\ref{corb5}.
\end{proof}

\subsection{Uniform Integrability}\label{s2.4}

The previously established estimates guarantee that there is no unlimited growth for small sizes nor escape towards large sizes during the evolution of $(f_j)_{j\ge 2}$. To achieve weak compactness in $X_{m_0}$, it remains to prevent concentration near a finite size. For that purpose, we recall that, since $f^{in}\in X_{m_0}$, a refined version of the de la Vall\'ee Poussin theorem ensures that there is $\Phi\in \mathcal{C}_{VP,\infty}$ depending only on $f^{in}$ such that 
\begin{equation}
\mathcal{I} := \int_0^\infty x^{m_0} \Phi(f^{in}(x))\ \mathrm{d}x < \infty\ , \label{b12}
\end{equation}
see \cite{BLLxx, Laur15, Le77, dlVP15}.

\begin{lemma}\label{lemb7}
\refstepcounter{NumConst}\label{cst4}
Let $T>0$ and $R>1$. There is $C_{\ref{cst4}}(T,R)>0$ depending on $T$ and $R$ such that
\begin{equation*}
\int_0^R x^{m_0} \Phi(f_j(t,x))\ \mathrm{d}x \le C_{\ref{cst4}}(T,R)\ , \qquad t\in [0,T]\ ,
\end{equation*} 
the function $\Phi$ being defined in \eqref{b12}.
\end{lemma}

\begin{proof}
We combine arguments from \cite{LaMi02b} with the subadditivity of $x\mapsto x^{m_0}$. Let $T>0$, $R>1$, and $t\in [0,T]$. On the one hand, 
\begin{align*}
I_{1,j}(R,t) & := \int_0^R x^{m_0} \Phi'(f_j(t,x)) \mathcal{C}_j f_j(t,x)\ \mathrm{d}x \\
& \le \frac{1}{2} \int_0^R \int_0^x x^{m_0} K(x-y,y) \Phi'(f_j(t,x)) f_j(t,x-y) f_j(t,y) \mathrm{d}y\mathrm{d}x\ ,
\end{align*}
and, by Fubini's theorem,
\begin{equation*}
I_{1,j}(R,t) \le \frac{1}{2} \int_0^R \int_0^{R-y} (x+y)^{m_0} K(x,y) \Phi'(f_j(t,x+y)) f_j(t,x) f_j(t,y) \mathrm{d}x\mathrm{d}y\ .
\end{equation*}
Owing to the subadditivity of $x\mapsto x^{m_0}$ and the symmetry of $K$, we further obtain
\begin{align*}
I_{1,j}(R,t) & \le \frac{1}{2} \int_0^R \int_0^{R-y} x^{m_0} K(x,y) \Phi'(f_j(t,x+y)) f_j(t,x) f_j(t,y) \mathrm{d}x\mathrm{d}y \\
& \qquad + \frac{1}{2} \int_0^R \int_0^{R-y} y^{m_0} K(x,y) \Phi'(f_j(t,x+y)) f_j(t,x) f_j(t,y) \mathrm{d}x\mathrm{d}y \\
& \le \frac{1}{2} \int_0^R \int_0^{R-x} y^{m_0} K(x,y) \Phi'(f_j(t,x+y)) f_j(t,x) f_j(t,y) \mathrm{d}y\mathrm{d}x \\
& \qquad + \frac{1}{2} \int_0^R \int_0^{R-x} y^{m_0} K(x,y) \Phi'(f_j(t,x+y)) f_j(t,x) f_j(t,y) \mathrm{d}y\mathrm{d}x \\
& \le  \int_0^R \int_0^{R-x} y^{m_0} K(x,y) \Phi'(f_j(t,x+y)) f_j(t,x) f_j(t,y) \mathrm{d}y\mathrm{d}x\ .
\end{align*}
Since $\Phi\in \mathcal{C}_{VP,\infty}$, we infer from \eqref{x1} with $r=f_j(t,x+y)$ and $s=f_j(t,y)$ that
\begin{align*}
I_{1,j}(R,t) & \le \int_0^R \int_0^{R-x} y^{m_0} K(x,y) \left[ \Phi(f_j(t,x+y)) + \Phi(f_j(t,y)) \right] f_j(t,x) \mathrm{d}y\mathrm{d}x \\
& \le \int_0^R \int_x^R (y-x)^{m_0} K(x,y-x) \Phi(f_j(t,y)) f_j(t,x) \mathrm{d}y\mathrm{d}x \\
& \qquad + \int_0^R \int_0^{R-x} y^{m_0} K(x,y) \Phi(f_j(t,y)) f_j(t,x) \mathrm{d}y\mathrm{d}x\ .
\end{align*}
We finally use \eqref{a7b} to conclude that
\begin{equation}
I_{1,j}(R,t) \le 2 L_R M_{m_0}(f_j(t)) \int_0^R y^{m_0} \Phi(f_j(t,y))\ \mathrm{d}y\ . \label{b13}
\end{equation}

On the other hand, we fix $p_0>1$ satisfying
\begin{equation*}
1 < p_0 < \frac{m_0+1}{-\nu} \le 1 + \frac{m_0}{-\nu} < 2\ ,
\end{equation*}
which is possible as $m_0+1>-\nu$. Using again Fubini's theorem,
\begin{align*}
I_{2,j}(R,t) & := \int_0^R x^{m_0} \Phi'(f_j(t,x)) \mathcal{F}_j f_j(t,x)\ \mathrm{d}x \\
& \le \int_0^R \int_x^\infty x^{m_0} a(y) \Phi'(f_j(t,x)) b_\nu(x,y) f_j(t,y) \mathrm{d}y\mathrm{d}x \\
& \le \int_0^\infty a(y) f_j(t,y) \int_0^{\min\{y,R\}} x^{m_0} b_\nu(x,y) \Phi'(f_j(t,x))\ \mathrm{d}x \mathrm{d}y \ .
\end{align*}
Since $\Phi\in \mathcal{C}_{VP,\infty}$, it follows from \eqref{a6}, Definition~\ref{defb1}~(c) with $p=p_0$, and \eqref{x1} with $r=f_j(t,x)$ and $s=x^\nu$ that
\begin{align*}
I_{2,j}(R,t) & \le (\nu+2) \int_0^\infty a(y) y^{-\nu-1} f_j(t,y) \int_0^{\min\{y,R\}} x^{m_0} \left[ \Phi(x^\nu) + \Phi(f_j(t,x)) \right]\ \mathrm{d}x \mathrm{d}y \\
& \le (\nu+2) S_{p_0}(\Phi) \int_0^\infty a(y) y^{-\nu-1} f_j(t,y) \int_0^y x^{m_0+\nu p_0}\ \mathrm{d}x \mathrm{d}y \\
& \qquad + (\nu+2) \int_0^\infty a(y) y^{-\nu-1} f_j(t,y) \int_0^R x^{m_0} \Phi(f_j(t,x))\ \mathrm{d}x \mathrm{d}y \\
& \le \frac{(\nu+2) S_{p_0}(\Phi)}{m_0+\nu p_0+1}  \int_0^\infty a(y) y^{m_0+\nu(p_0-1)} f_j(t,y)\ \mathrm{d}y \\
& \qquad + (\nu+2) \left( \int_0^\infty a(y) y^{-\nu-1} f_j(t,y)\ \mathrm{d}y \right) \int_0^R x^{m_0} \Phi(f_j(t,x))\ \mathrm{d}x\ .
\end{align*}
Since $m_0+\nu(p_0-1)$ and $-\nu-1$ both belong to $[0,1)$ and $m_0+\nu p_0+1>0$, we deduce from \eqref{a8} and Lemma~\ref{lemb6}  that
\begin{align*}
I_{2,j}(R,t) & \le C \left[ A_1 \int_0^1 y^{2m_0+1 + \nu p_0} f_j(t,y)\ \mathrm{d}y + P_{m_0+\nu(p_0-1)}(t) \right] \\
& \qquad + \left( A_1 \int_0^1 y^{m_0} f_j(t,y)\ \mathrm{d}y + P_{-\nu-1,j}(t) \right) \int_0^R x^{m_0} \Phi(f_j(t,x))\ \mathrm{d}x \\
& \le C \left[ A_1 C_{\ref{cst3}}(T) + P_{m_0+\nu(p_0-1)}(t) \right] \\
& \qquad + \left[ A_1 C_{\ref{cst3}}(T) + P_{-\nu-1,j}(t) \right] \int_0^R x^{m_0} \Phi(f_j(t,x))\ \mathrm{d}x \ .
\end{align*}
Combining the previous inequality with \eqref{b5}, \eqref{b13}, and Lemma~\ref{lemb6} gives
\begin{align*}
\frac{\mathrm{d}}{\mathrm{d}t} \int_0^R x^{m_0} \Phi(f_j(t,x))\ \mathrm{d}x & = I_{1,j}(R,t) + I_{2,j}(R,t) \\
& \hspace{-2cm} \le \left[ (2L_R+A_1) C_{\ref{cst3}}(T) + (\nu+2) P_{-\nu-1,j}(t) \right] \int_0^R x^{m_0} \Phi(f_j(t,x))\ \mathrm{d}x \\
& \hspace{-2cm} \qquad + C \left[ A_1 C_{\ref{cst3}}(T) + P_{m_0+\nu(p_0-1)}(t) \right]\ .
\end{align*}
Owing to Corollary~\ref{corb5}, we are in a position to apply Gronwall's lemma to complete the proof.
\end{proof}

Owing to \eqref{b7}, Lemma~\ref{lemb7}, and the superlinearity of $\Phi$ at infinity, see Definition~\ref{defb1}~(b) , we deduce from Dunford-Pettis' theorem that, for all $T>0$, there is a weakly compact subset $\mathcal{K}(T)$ of $X_{m_0}$ depending on $T$ such that
\begin{equation}
f_j(t) \in \mathcal{K}(T)\ , \qquad t\in [0,T]\ . \label{b14}
\end{equation}

\subsection{Time Equicontinuity}\label{s2.5}

Having established the (weak) compactness of the sequence $(f_j)_{j\ge 2}$ with respect to the size variable $x$, we now focus on the time variable. Our aim being to apply a variant of Arzel\`a-Ascoli's theorem, we study time equicontinuity properties of $(f_j)_{j\ge 2}$ in the next result.

\begin{lemma}\label{lemb8}
\refstepcounter{NumConst}\label{cst5} \refstepcounter{NumConst}\label{cst6}
Let $T>0$ and $R>1$. There are $C_{\ref{cst5}}(T,R)>0$ depending on $T$ and $R$ and $C_{\ref{cst6}}(T)>0$ depending on $T$ such that, for $t_1\in [0,T)$ and $t_2\in (t_1,T]$,
\begin{equation*}
\int_0^R x^{m_0} |f_j(t_2,x) - f_j(t_1,x)|\ \mathrm{d}x \le C_{\ref{cst5}}(R,T) (t_2-t_1) + C_{\ref{cst6}}(T) R^{(m_0-1)/2}\ .
\end{equation*}
\end{lemma}

\begin{proof}
Let $t\in [0,T]$. First, by Fubini's theorem,
\begin{align*}
I_{3,j}(R,t) & := \int_0^R x^{m_0} |\mathcal{C}_j f_j(t,x)|\ \mathrm{d}x \\
& \le \frac{1}{2} \int_0^R \int_0^{R-x} (x+y)^{m_0} K(x,y) f_j(t,x) f_j(t,y)\ \mathrm{d}y \mathrm{d}x \\
& \qquad + \int_0^R \int_0^R x^{m_0} K(x,y) f_j(t,x) f_j(t,y)\ \mathrm{d}y \mathrm{d}x \\
& \qquad + \int_0^R \int_R^\infty x^{m_0} K(x,y) f_j(t,x) f_j(t,y)\ \mathrm{d}y \mathrm{d}x \ .
\end{align*}
Owing to \eqref{a7} and the elementary inequality $(x+y)^{m_0}\le 2 \max\{x,y\}^{m_0}$, $(x,y)\in (0,\infty)^2$, we further obtain
\begin{align*}
I_{3,j}(R,t) & \le 2 L_R \int_0^R \int_0^R x^{m_0} y^{m_0} f_j(t,x) f_j(t,y)\ \mathrm{d}y \mathrm{d}x \\
& \qquad + K_0 \int_0^R \int_R^\infty x^{m_0} (2+x+y) f_j(t,x) f_j(t,y)\ \mathrm{d}y \mathrm{d}x \\
& \le 2 L_R M_{m_0}(f_j(t))^2 + 4 K_0 M_{m_0}(f_j(t)) M_1(f_j(t))\ ,
\end{align*}
hence, thanks to \eqref{b7} and Lemma~\ref{lemb6}, 
\begin{equation}
I_{3,j}(R,t) \le 2 L_R C_{\ref{cst3}}(T)^2 + 4 \varrho K_0 C_{\ref{cst3}}(T)\ . \label{b15}
\end{equation}
Next, using once more Fubini's theorem,
\begin{align*}
I_{4,j}(R,t) & := \int_0^R x^{m_0} |\mathcal{F}_j f_j(t,x)|\ \mathrm{d}x \\
& \le \int_0^R x^{m_0} a(x) f_j(t,x)\ \mathrm{d}x + \int_0^R a(y) f_j(t,y) \int_0^y x^{m_0} b_\nu(x,y)\ \mathrm{d}x \mathrm{d}y \\
& \qquad + \int_R^\infty a(y) f_j(t,y) \int_0^R x^{m_0} b_\nu(x,y)\ \mathrm{d}x \mathrm{d}y \ .
\end{align*}
We now infer from \eqref{a6}, \eqref{a8}, and Lemma~\ref{lemb6} that
\begin{align*}
I_{4,j}(R,t) & \le A_R R^{m_0+\nu+1} M_{m_0}(f_j(t)) + \frac{\nu+2}{m_0+\nu+1} \int_0^R a(y) y^{m_0} f_j(t,y)\ \mathrm{d}y \\
& \qquad + \frac{\nu+2}{m_0+\nu+1} R^{m_0+\nu+1} \int_R^\infty a(y) y^{-\nu-1} f_j(t,y)\ \mathrm{d}y \\
& \le C A_R R^{m_0+\nu+1} C_{\ref{cst3}}(T) + C R^{(m_0-1)/2} \int_R^\infty a(y) y^{(m_0+1)/2} f_j(t,y)\ \mathrm{d}y \ ,
\end{align*}
hence
\begin{equation}
I_{4,j}(R,t) \le C A_R R^{m_0+\nu+1} C_{\ref{cst3}}(T) + C R^{(m_0-1)/2} P_{(m_0+1)/2,j}(t)\ . \label{b16}
\end{equation}
It then follows from \eqref{b5}, \eqref{b15}, and \eqref{b16}  that
\begin{align*}
\int_0^R x^{m_0} |\partial_t f_j(t,x)|\ \mathrm{d}x & \le I_{3,j}(R,t)  + I_{4,j}(R,t) \\
& \le C_{\ref{cst5}}(R,T) + C R^{(m_0-1)/2} P_{(m_0+1)/2,j}(t)\ .
\end{align*} 
Consequently, since $(m_0+1)/2\in (0,1)$, Corollary~\ref{corb5} entails 
\begin{align*}
\int_0^R x^{m_0} |f_j(t_2,x) - f_j(t_1,x)|\ \mathrm{d}x & \le \int_{t_1}^{t_2} \int_0^R x^{m_0} |\partial_t f_j(t,x)|\ \mathrm{d}x\mathrm{d}t \\ 
& \le C_{\ref{cst5}}(R,T) (t_2-t_1) + C_{\ref{cst6}}(T) R^{(m_0-1)/2}\ ,
\end{align*}
and the proof is complete.
\end{proof}

Combining \eqref{b7} and Lemma~\ref{lemb8} allows us to improve the equicontinuity with respect to time of the sequence $(f_j)_{j\ge 2}$.

\begin{corollary}\label{corb9}
Let $T>0$. There is a function $\omega(T,\cdot): [0,\infty)\to [0,\infty)$ depending on $T$ such that
\begin{equation}
\lim_{s\to 0} \omega(T,s)=0\ , \label{b17}
\end{equation}
and, for $t_1\in [0,T)$ and $t_2\in (t_1,T]$,
\begin{equation*}
\int_0^\infty x^{m_0} |f_j(t_2,x) - f_j(t_1,x)|\ \mathrm{d}x \le \omega(T,t_2-t_1)\ .
\end{equation*}
\end{corollary}

\begin{proof}
Let $0\le t_1 < t_2\le T$ and $R>1$. By \eqref{b7} and Lemma~\ref{lemb8},
\begin{align*}
\int_0^\infty x^{m_0} |f_j(t_2,x) - f_j(t_1,x)|\ \mathrm{d}x & \le \int_0^R x^{m_0} |f_j(t_2,x) - f_j(t_1,x)|\ \mathrm{d}x \\
& \qquad + R^{m_0-1} \int_R^\infty x [f_j(t_1,x) + f_j(t_2,x)]\ \mathrm{d}x \\
& \le C_{\ref{cst5}}(R,T) (t_2-t_1) + C_{\ref{cst6}}(T) R^{(m_0-1)/2} + 2 \varrho R^{m_0-1}\ .
\end{align*}
Introducing 
\begin{equation*}
\omega(T,s) := \inf_{R>1}\left\{ C_{\ref{cst5}}(R,T) s + C_{\ref{cst6}}(T) R^{(m_0-1)/2} + 2 \varrho R^{m_0-1} \right\}\ , \qquad s\ge 0\ ,
\end{equation*}
we deduce from the previous inequality which is valid for all $R>1$ that
\begin{equation*}
\int_0^\infty x^{m_0} |f_j(t_2,x) - f_j(t_1,x)|\ \mathrm{d}x \le \omega(T,t_2-t_1)
\end{equation*}
and observe that $\omega(T,\cdot)$ enjoys the property \eqref{b17} as $m_0<1$.
\end{proof}

\subsection{Convergence}\label{s2.6}

According to a variant of Arzel\`a-Ascoli's theorem \cite[Theorem~A.3.1]{Vrab03}, it readily follows from \eqref{b14} and Corollary~\ref{corb9} that the sequence $(f_j)_{j\ge 2}$ is relatively compact in $C([0,T];X_{m_0,w})$ for all $T>0$. Using a diagonal process, we construct a subsequence of $(f_j)_{j\ge 2}$ (not relabeled) and $f\in C([0,\infty);X_{m_0,w})$ such that, for all $T>0$,
\begin{equation}
f_j \longrightarrow f \;\text{ in }\; C([0,T];X_{m_0,w})\ . \label{b18}
\end{equation}
A first consequence of \eqref{b18} is that $f(t)$ is a non-negative function for all $t\ge 0$ and that $f(0)=f^{in}$. It further follows from Lemma~\ref{lemb4},  the superlinearity of $W$ for large sizes, and \eqref{b18} that the latter can be improved to 
\begin{equation}
f_j \longrightarrow f \;\text{ in }\; C([0,T];X_{1,w})\ . \label{b19}
\end{equation}
A straightforward consequence of \eqref{b7} and \eqref{b19} is the mass conservation
\begin{equation*}
M_1(f(t)) = M_1(f^{in})\ , \qquad t\ge 0\ . 
\end{equation*}
Owing to \eqref{a7}, \eqref{a8}, Corollary~\ref{corb5}, \eqref{b18}, and \eqref{b19}, it is by now a standard argument to pass to the limit in \eqref{b6} and deduce that $f$ is a weak solution on $[0,\infty)$ to \eqref{a1} in the sense of Definition~\ref{defb0}, see \cite{BLLxx, LaMi02b, Stew89} for instance. It is worth pointing out that the behaviour of the fragmentation term for large sizes is controlled by Corollary~\ref{corb5} which guarantees that, for all $T>0$, 
\begin{equation*}
\lim_{R\to\infty} \sup_{j\ge 2} \int_0^T \int_R^\infty a(x) f_j(s,x)\ \mathrm{d}x \mathrm{d}s =0\ ,
\end{equation*}
and thereby allows us to perform the limit $j\to\infty$ in the fragmentation term. We have thus completed the proof of Theorem~\ref{thmi1}, except for the propagation of moments of higher order which is proved in Proposition~\ref{propb10} in the next section.

\subsection{Higher Moments}\label{s2.7}

We supplement the above analysis with the study of the evolution of algebraic moments of any order. Let $f$ be the mass-conserving weak solution on $[0,\infty)$ to \eqref{a1} constructed in the previous section.

\begin{proposition}\label{propb10}
Consider $m>1$ and assume further that $f^{in}\in X_m$. Then $f\in L^\infty(0,T;X_m)$ for all $T>0$.
\end{proposition}

\begin{proof}
Let $T>0$, $t\in (0,T)$, and $j\ge 2$. Setting $\vartheta_m(x):=x^m$ for $x\in (0,\infty)$, it readily follows from \eqref{a6} that
\begin{equation*}
N_{\vartheta_m}(y) = \frac{m-1}{m+\nu+1} y^m \ge 0\ , \qquad y>0\ ,
\end{equation*}
\refstepcounter{NumConst}\label{cst7}
while there is $C_{\ref{cst7}}(m)>0$ depending only on $m$ such that
\begin{equation*}
(x+y) \chi_{\vartheta_m}(x,y) \le C_{\ref{cst7}}(m) \left( x^m y + x y^m \right)\ , \qquad (x,y)\in (0,\infty)^2\ ,
\end{equation*} 
by \cite[Lemma~2.3~(ii)]{Carr92}. Consequently, by \eqref{a7},
\begin{equation*}
\chi_{\vartheta_m}(x,y) K(x,y) \le 2 L_1 \min\{x,y\}^{m_0} \max\{x,y\}^m \le 2 L_1 (xy)^{m_0}\ , \qquad (x,y)\in (0,1)^2\ ,
\end{equation*} 
and 
\begin{align*}
\chi_{\vartheta_m}(x,y) K(x,y) & \le K_0 (2+x+y) \chi_{\vartheta_m}(x,y) \le 3 K_0 (x+y) \chi_{\vartheta_m}(x,y) \\
& \le 3 K_0 C_{\ref{cst7}}(m) \left( x^m y + x y^m \right)\ , \qquad (x,y)\in (0,\infty)^2\setminus (0,1)^2\ .
\end{align*}
We then infer from \eqref{b6} with $\vartheta=\vartheta_m$ and the previous inequalities that
\begin{equation*}
\frac{\mathrm{d}}{\mathrm{d}t} M_m(f_j(t)) \le L_1 M_{m_0}(f_j(t))^2 + 3 K_0 C_{\ref{cst7}}(m) M_1(f_j(t)) M_m(f_j(t))\ .
\end{equation*}
Recalling \eqref{b7} and Lemma~\ref{lemb6}, we end up with
\begin{equation*}
\frac{\mathrm{d}}{\mathrm{d}t} M_m(f_j(t)) \le L_1 C_{\ref{cst3}}(T)^2 + 3 K_0 C_{\ref{cst7}}(m) \varrho M_m(f_j(t))\ ,
\end{equation*}
and integrate the previous differential inequality to deduce that
\begin{equation}
M_m(f_j(t)) \le e^{3 K_0 C_{\ref{cst7}}(m) \varrho t} \left[ M_m(f_j^{in}) + \frac{L_1 C_{\ref{cst3}}(T)^2}{3 K_0 C_{\ref{cst7}}(m) \varrho}\right]\ , \qquad t\in [0,T]\ . \label{b21}
\end{equation}
Since $M_m(f_j^{in})\le M_m(f^{in})<\infty$, Lemma~\ref{lemb8} readily follows from \eqref{b21} after letting $j\to\infty$ with the help of \eqref{b19} and Fatou's lemma.
\end{proof}

\section{Uniqueness}\label{s3}

\begin{proof}[Proof of Theorem~\ref{thmi2}]
Consider two weak solutions $f_1$ and $f_2$  on $[0,\infty)$ to \eqref{a1}-\eqref{a2} on $[0,\infty)$ enjoying the properties listed in Theorem~\ref{thmi2}. We set $F := f_1-f_2$, $\sigma := \mathrm{sign}(f_1-f_2)$, $\xi(x) := \max\{x^{m_0}, x^{1+\delta} \}$, and
\begin{equation*}
\Xi(t,x,y) := \xi(x+y) \sigma(t,x+y) - \xi(x) \sigma(t,x) - \xi(y) \sigma(t,y)
\end{equation*} 
for $(t,x,y)\in (0,\infty)^3$. Arguing as in the proof of \cite[Theorem~2.9]{EMRR05} and \cite{Stew90b}, we deduce from \eqref{a1} that
\begin{align*}
& \frac{\mathrm{d}}{\mathrm{d}t} \int_0^\infty \xi(x) |F(t,x)|\ \mathrm{d}x \\ 
& \qquad = \frac{1}{2} \int_0^\infty \int_0^\infty K(x,y) \Xi(t,x,y) (f_1+f_2)(t,y) F(t,x)\ \mathrm{d}y \mathrm{d}x \\
& \qquad \qquad + \int_0^\infty a(y) F(t,y) \left[ \int_0^y \xi(x) \sigma(t,x) b_\nu(x,y)\ \mathrm{d}x - \xi(y) \sigma(t,y) \right]\ \mathrm{d}y\ .
\end{align*}
Since
\begin{align*}
\Xi(t,x,y) F(t,x) & = \xi(x+y) \sigma(t,x+y) F(t,x) - \xi(x) |F(t,x)| - \xi(y) \sigma(t,y) F(t,x) \\
& \le \xi(x+y) |F(t,x)| - \xi(x) |F(t,x)| + \xi(y) |F(t,x)|
\end{align*}
and
\begin{align*}
& F(t,y) \left[ \int_0^y \xi(x) \sigma(t,x) b_\nu(x,y)\ \mathrm{d}x - \xi(y) \sigma(t,y) \right] \\
& \qquad \le \int_0^y \xi(x) b_\nu(x,y)\ \mathrm{d}x |F(t,y)| - \xi(y) |F(t,y)| \\
& \qquad \le - N_\xi(y) |F(t,y)|\ ,
\end{align*}
we obtain that
\begin{align}
\frac{\mathrm{d}}{\mathrm{d}t} \int_0^\infty \xi(x) |F(t,x)|\ \mathrm{d}x & \le \int_0^\infty \int_0^\infty K(x,y) \Xi_0(x,y) (f_1+f_2)(t,y) |F(t,x)|\ \mathrm{d}y \mathrm{d}x \nonumber \\
& \qquad - \int_0^\infty a(y) N_\xi(y) |F(t,y)|\ \mathrm{d}y\ , \label{b22}
\end{align}
with
\begin{equation*}
\Xi_0(x,y) := \xi(x+y) -\xi(x) + \xi(y)\ , \qquad (x,y)\in (0,\infty)^2\ .
\end{equation*}
On the one hand, we infer from \eqref{a7}, \eqref{u0}, and the subadditivity of $x\mapsto x^{m_0}$ and $x\mapsto x^\delta$ that:
\begin{itemize}
\item[(c1)] for $(x,y)\in (0,1)^2$,
\begin{align*}
K(x,y) \Xi_0(x,y) & \le L_1 \min\{x,y\}^{m_0} \left[ (x+y)^{m_0} - x^{m_0} + y^{m_0} \right] \\ 
& \le 2 L_1 \min\{x,y\}^{m_0} y^{m_0} \le 2 L_1 \xi(x) y^{m_0}\ ;
\end{align*}
\item[(c2)] for $(x,y)\in (0,1)\times (1,\infty)$, 
\begin{align*}
K(x,y) \Xi_0(x,y) & \le K_1 x^{m_0} y \left[ x^{1+\delta} + (1+\delta) y (x+y)^\delta - x^{m_0} + y^{1+\delta} \right] \\
& \le K_1 x^{m_0} y \left[ (1+\delta) 2^\delta y^{1+\delta} + y^{1+\delta} \right] \\
& \le K_1 \left[ 1 + (1+\delta)2^\delta \right] \xi(x) y^{2+\delta}\ ;
\end{align*}
\item[(c3)] for $(x,y)\in (1,\infty)\times (0,1)$,
\begin{align*}
K(x,y) \Xi_0(x,y) & \le K_0 (2+x+y) \left[ (1+\delta) y (x+y)^\delta + y^{m_0} \right] \\ 
& \le 4 K_0 x \left[ (1+\delta) 2^\delta x^\delta y + x^\delta y^{m_0} \right] \\
& \le 4 K_0 \left[ 1 + (1+\delta)2^\delta \right] \xi(x) y^{m_0}\ ;
\end{align*}
\item[(c4)] for $(x,y)\in (1,\infty)^2$, 
\begin{align*}
K(x,y) \Xi_0(x,y) & \le K_0 (2+x+y) \left[ (1+\delta) y (x+y)^\delta + y^{1+\delta} \right] \\ 
& \le 2 K_0 (x+y) \left[ (1+\delta) y x^\delta + (2+\delta) y^{1+\delta} \right] \\
& \le 2(2+\delta) K_0 \left( x^{1+\delta} y + x y^{1+\delta} + x^\delta y^2 + y^{2+\delta} \right)\\ 
& \le 8(2+\delta) K_0 \xi(x) y^{2+\delta}\ .
\end{align*}
\end{itemize} 
On the other hand, owing to \eqref{a6} and \eqref{a8},
\begin{itemize}
\item[(f1)] for $y\in (0,1)$, 
\begin{equation*}
- a(y) N_\xi(y) = \frac{1-m_0}{\nu+1+m_0}y^{m_0} a(y) \le \frac{A_1}{\nu+1+m_0} \xi(y) \ ;
\end{equation*}
\item[(f2)] for $y>1$, 
\begin{align*}
- N_\xi(y) & = \frac{\nu+2}{\nu+1+m_0} y^{-\nu-1} + \frac{\nu+2}{\nu+2+\delta} \left( y^{1+\delta} - y^{-\nu-1} \right) - y^{1+\delta} \\
& = \frac{(\nu+2)(1+\delta-m_0)}{(\nu+1+m_0)(\nu+2+\delta)} y^{-\nu-1} - \frac{\delta}{\nu+2+\delta} y^{1+\delta} \\
& = \frac{\delta}{\nu+2+\delta} y^{-\nu-1} \left( Y^{2+\delta+\nu} - y^{2+\delta+\nu} \right) \ ,
\end{align*}
with
\begin{equation*}
Y^{2+\delta+\nu} := \frac{(\nu+2)(1+\delta-m_0)}{\delta(\nu+1+m_0)}\ .
\end{equation*}
Either $y\ge \max\{1,Y\}$ and 
\begin{equation*}
- a(y) N_\xi(y) \le \frac{\delta}{\nu+2+\delta} a(y) y^{-\nu-1} \left( Y^{2+\delta+\nu} - y^{2+\delta+\nu} \right) \le 0\ ,
\end{equation*} 
or $y\in (1,\max\{1,Y\})$ and
\begin{align*}
- a(y) N_\xi(y) & \le \frac{\delta}{\nu+2+\delta} a(y) y^{-\nu-1} Y^{2+\delta+\nu} \\
& \le A_Y Y^{2+\delta+\nu} y^{m_0} \le A_Y Y^{2+\delta+\nu} \xi(y)\ .
\end{align*}
\end{itemize}

Collecting the estimates in (c1)--(c4) and (f1)-(f2), we infer from \eqref{b22} that there is a positive constant $\kappa>0$ depending only on $L_1$, $K_1$, $\delta$, $K_0$, $\nu$, $m_0$, and $a$ such that
\begin{align*}
& \frac{\mathrm{d}}{\mathrm{d}t} \int_0^\infty \xi(x) |F(t,x)|\ \mathrm{d}x \\ 
& \qquad \le \kappa \int_0^\infty \int_0^\infty \xi(x) \left( y^{m_0} + y^{2+\delta} \right) (f_1+f_2)(t,y) |F(t,x)|\ \mathrm{d}y\mathrm{d}x \\
& \qquad\qquad + \kappa \int_0^\infty \xi(y) |F(t,y)|\ \mathrm{d}y \\
& \qquad \le \kappa \left[ 1 + M_{m_0}((f_1+f_2)(t)) + M_{2+ \delta}((f_1+f_2)(t)) \right] \int_0^\infty \xi(y) |F(t,y)|\ \mathrm{d}y\ .
\end{align*}
Since $M_{m_0}(f_i)$ and $M_{2+\delta}(f_i)$ both belong to $L^\infty(0,T)$ for $i\in\{1,2\}$ and $f_1(0)=f_2(0)=f^{in}$, we use Gronwall's lemma to complete the proof.
\end{proof}


\bibliographystyle{siam}
\bibliography{NonIntegrableFragmentDistribution}

\end{document}